\documentclass[a4paper,11pt,notitlepage,reqno]{article}
\usepackage[english]{babel}
\usepackage[latin1]{inputenc}
\usepackage{amssymb,amsthm,amsmath}
\usepackage{mathtools}
\usepackage{mathrsfs}
\usepackage{enumerate}
\usepackage{authblk}
\usepackage{stackrel}
\usepackage{eucal}
\usepackage{color}
\usepackage{xcolor}
\usepackage{graphics}
\usepackage{subcaption}
\usepackage{geometry}
\usepackage{hyperref}

\usepackage{xcolor}
\hypersetup{
	colorlinks,
	linkcolor={red!60!black},
	citecolor={green!60!black},
	urlcolor={blue!60!black}
}

\numberwithin{equation}{section}
\newcommand{\R}{{\mathbb R}}

\newtheorem*{remark*}{Remark}

\newtheorem{prop}{Proposition}
\newtheorem{thm}{Theorem}
\newtheorem{lem}{Lemma}
\newtheorem{set}{Setting}
\newtheorem{prob}{Problem}

\theoremstyle{definition}
\newtheorem{rem}{Remark}
\newtheorem{ex}{Example}

\begin{document}
\title{On a functional-differential equation \\
			with quasi-arithmetic mean value}
\author{Shokhrukh Ibragimov
	\\
	\small 	Department of Mathematics, National University of Uzbekistan\\
	Almazar street Universitetskaya 4, 100174 Tashkent, Uzbekistan\\
	E-mail: shohruh.i.95@gmail.com
}

\maketitle

\begin{abstract}
\noindent
In this paper we describe all differentiable functions $\varphi, \psi  \colon E\to\R$ satisfying the functional-differential equation
  \begin{equation}
  [\varphi(y)-\varphi(x)]\psi '\bigl(h(x,y)\bigr)=[\psi(y)-\psi(x)]\varphi '\bigl(h(x,y)\bigr),
  \end{equation}
for all $x,y\in E$, $x<y$, where $E \subseteq \R$ is a nonempty open interval, $h(\cdot, \cdot)$ is a quasi-arithmetic mean, i.e. $h(x,y)=H^{-1}(\alpha H (x)+\beta H (y)), x,y\in E$, for some differentiable and strictly monotone function $H \colon E \to H(E)$ and fixed $\alpha, \beta\in (0,1)$ with $\alpha+\beta=1$.
\end{abstract}

\maketitle

\section{Introduction}
Given a nonempty open interval $E\subseteq \R$ and differentiable functions $\varphi,\psi \colon E \to \R$, the Cauchy Mean Value Theorem (MVT) states that, for any interval $(a,b)\subset E$ there exists a point $c$ in $(a,b)$ such that
\begin{equation} \label{Cauchy}
[\varphi(b)-\varphi(a)]\,\psi '(c) = [\psi(b)-\psi(a)]\,\varphi'(c).
\end{equation}
A particular case is the Lagrange MVT, in which case $\psi$ is the identity function and hence, \eqref{Cauchy} reads as
\begin{equation}\label{Lagrange}
  \varphi(b)-\varphi(a)=(b-a)\varphi'(c).
\end{equation}
It is an interesting question to ask for which $\varphi$ and $\psi$ the mean value $c$ in \eqref{Cauchy} or~\eqref{Lagrange} depends on the endpoints $a$ and $b$ in a prescribed way. This is the subject of this note. More precisely, we are concerned with the following problem.
\begin{prob}\label{the-problem}
Let $E \subseteq \R$ be a nonempty open interval, let  $H \colon E \to H(E)$ be a differentiable and strictly monotone function and let $h \colon E \times E \to \R$ satisfy for all $x, y \in E$ that $h(x,y)=H^{-1}(\alpha H (x)+\beta H (y))$, where $\alpha, \beta \in  (0,1)$ satisfy $\alpha+\beta=1$.
Find all pairs $(\varphi, \psi)$ of differentiable functions $\varphi, \psi \colon E\to\R$ which satisfy for all $x, y \in E$, with $x < y$, that
\begin{equation}
\label{eqn:cauchy.gen}
  [\varphi(y)-\varphi(x)]\psi '\bigl(h(x,y)\bigr)=[\psi(y)-\psi(x)]\varphi '\bigl(h(x,y)\bigr).
\end{equation}
\end{prob}

For the case of the Lagrange MVT (i.e.~when $\psi(x)=x$) with $H(x)=x$ and $\alpha=\beta=\frac{1}{2}$, this problem was considered first by Haruki~\cite{Haruki} and independently by Acz\'{e}l~\cite{Aczel}, who showed that the quadratic functions are the only solutions to \eqref{eqn:cauchy.gen}. This particular case serve as
a starting point for various functional equations, see e.g. Sahoo and Riedel~\cite{Sahoo}. More general functional equations have been considered even in the abstract setting of groups by several researchers including Kannappan~\cite{Kannappan}, Ebanks~\cite{Ebanks} and Fechner \& Gselmann~\cite{Fechner-Gselmann}. Moreover, the result of Acz\'{e}l and Haruki has been generalized for higher order  Taylor expansion by Sablik~\cite{Sablik}.

The 
functional-differential equation \eqref{eqn:cauchy.gen} was solved by Balogh, Ibrogimov \& Mityagin~\cite{BIM} for the case $H(x)=x, E=\R$, under the assumption that $\varphi$ and $\psi$ are three times differentiable functions. It is worth to mention that the method of \cite{BIM} applies to the case of arbitrary open interval $E$ without much additional effort. Recently, Lukasik~\cite{Lukasik} has combined the method of \cite{BIM} together with an indirect approach to provide all solutions of \eqref{eqn:cauchy.gen} for the case of $H(x)=x$ and an arbitrary open interval $E\subseteq \R$ by only requiring the differentiability of the unknown functions. For the general case, Kiss \& P\'{a}les~\cite{Pales2018} has provided all solutions of the equation \eqref{eqn:cauchy.gen} under the assumption that $\psi'$ does not vanish on $E$ and that $\frac{\varphi'}{\psi'}$ is invertible.

In this note we provide a complete solution to Problem~\ref{the-problem} by a different and self-contained approach. Our method is heavily inspired by that of \cite{BIM} and based on the tricks analogous to the ones of \cite{Pales2018} and \cite{Lukasik}. Our main result reads as follows.

\begin{thm}
  \label{thm:sym.cauchy.main}
Assume the setting of Problem~{\upshape\ref{the-problem}}. If $\varphi$ and $\psi$ solves the functional-differential equation \eqref{eqn:cauchy.gen}, then one of the following possibilities holds:

\begin{enumerate}[{\upshape(a)}]

\item $\{1, \varphi, \psi\}$ are linearly dependent on every open subinterval of $E$, where $\psi'$ does not vanish. Moreover, if $J=H(E)$ is a semi-infinite interval, then $\{1, \varphi, \psi\}$ are linearly dependent on $E$;

\item $\varphi,\psi \in \mathrm{span} \{1,H,H^2\}$ on $E$;

\item there exists a non-zero real number $\mu$ such that
\[
\varphi,\psi \in \mathrm{span} \{1,e^{\mu H}, e^{-\mu H}\}\quad\text{on}\quad E;
\]

\item there exists a non-zero real number $\mu$ such that
\[
\varphi,\psi \in \mathrm{span} \{1,\sin(\mu H), \cos(\mu H)\}\quad\text{on}\quad E.
\]
\end{enumerate}

\end{thm}

\begin{rem}
  It is easy to check that any of the cases (a)$-$(d) of the theorem indeed provides a solution to \eqref{eqn:cauchy.gen}.
\end{rem}
\begin{rem}
  For $H(x)=x^p$ with $p\in \R \backslash \{0\}$ and $E=[0,\infty)$, the function $h(\cdot,\cdot)$ reads as the power mean (or $p$-mean), i.e.
  \begin{equation*}
    h(x,y)=\bigl( \alpha x^p+\beta y^p \bigr)^{\frac{1}{p}}, \quad x,y\in [0,\infty),
  \end{equation*}
and Theorem~\ref{thm:sym.cauchy.main} solves the corresponding problem mentioned in \cite{BIM}.
\end{rem}

We emphasize that the even more general problem, where the derivatives of the unknown functions in \eqref{eqn:cauchy.gen} are replaced by other two unknown functions, is still an unsolved problem 
as mentioned in the Newsletter of the European Mathematical Society in 2016 (see \cite{rassias2016solved}, \cite{Sahoo}).

The rest of the note is organized as follows. In Section~\ref{sec2} we prove that the functions $\varphi$ and $\psi$ which solve Problem~\ref{the-problem} are either linearly dependent or infinitely differentiable on any subinterval of $E$ in which $\psi'$ does not vanish. In Section~\ref{sec3} we provide a result which will allow us to extend the linearly independency intervals to the whole set. In Section~\ref{sec4} we analyse the asymmetric case ($\alpha\neq\beta$). In Section~\ref{sec5} the symmetric case ($\alpha=\beta=1/2$) is considered and the proof of Theorem~\ref{thm:sym.cauchy.main} is completed.


\section{Infinite differentiability of unknown functions}
\label{sec2}

 We start with transforming Problem~\ref{the-problem} to a problem with a linear mean as follows. It is given that the function $H$ is strictly monotone and differentiable on $E$. Without loss of generality  we may assume that $H$ is strictly increasing. Hence, the inverse of $H$ is also strictly increasing, differentiable, and has a non-vanishing derivative on $J:=H(E)$. Substituting  $x \mapsto H^{-1}(a)$, $y \mapsto H^{-1}(b)$  in \eqref{eqn:cauchy.gen} and denoting  $F=\varphi \circ H^{-1}$ and $G=\psi \circ H^{-1}$, we get
\begin{equation*}
  [F(b)-F(a)]G'(\alpha a+\beta b)=[G(b)-G(a)]F'(\alpha a+\beta b)
\end{equation*}
for all $a,b \in J$. Therefore, Problem~\ref{the-problem} reduces to the following problem.

\begin{prob}\label{3}
Let $J \subseteq \R$ be a nonempty open interval and  let $\alpha, \beta\in (0,1)$ be fixed with $\alpha+\beta=1$. Find all pairs $(F,G)$ of differentiable functions $F,G \colon J \to \R$ which satisfy for all $a, b \in J$ that
\begin{equation}
\label{eq:prob2}
  [F(b)-F(a)]G'(\alpha a+\beta b)=[G(b)-G(a)]F'(\alpha a+\beta b).
\end{equation}
\end{prob}
In this section we present certain properties of the solutions of the functional-differential equation~\eqref{eq:prob2} on the interval where the derivative of $G$ does not vanish. For convenience, we assume the following setting throughout this section.

\begin{set}
	\label{setting}
Let $J \subseteq \R$ be a nonempty open interval, let $\alpha, \beta\in (0,1)$ be fixed with $\alpha+\beta=1$, let $F,G \colon J \to \R$ be differentiable functions, with derivatives $F' = f$, $G'=g$, satisfy for all $a, b \in J$ that
\begin{equation}\label{new-eq}
[F(b)-F(a)]g(\alpha a+\beta b)=[G(b)-G(a)]f(\alpha a+\beta b),
\end{equation}
let  $I \subseteq J$ be  a nonempty interval such that  for all $x\in I$ it holds that $g(x)\neq 0$, and let $v \colon I \to \R$ satisfy for all $x \in I$ that $v(x) = \frac{f(x)}{g(x)}$.
\end{set}

\begin{lem}
\label{lem:elem}
	For all $a, b \in I$ it holds that
	\begin{equation}\label{new-eq3}
	\beta g(a)(v(\alpha a+\beta b)-v(a))=\alpha g(b)(v(b)-v(\alpha a+\beta b)).
\end{equation}
\end{lem}
\begin{proof}
Using the transformation $\alpha a+\beta b \mapsto x$, $b-a \mapsto h$ the condition \eqref{new-eq} reads as
\begin{equation}\label{new-eq1}
[F(x+\alpha h)-F(x-\beta h)]g(x)=[G(x+\alpha h)-G(x-\beta h)]f(x)
\end{equation}
for all $x\in I$ and $h\in \R $ such that $x+\alpha h, x-\beta h\in I$. Differentiating both sides with respect to $h$ we obtain
\begin{equation}
[\alpha f(x+ \alpha h) + \beta f(x-\beta h) ] g(x) = [\alpha g(x+\alpha h) + \beta g(x -\beta h)] f(x)
\end{equation}
for all $x\in I$ and $h\in \R $ such that $x+\alpha h, x-\beta h\in I$. Applying the transformation  $h \mapsto b-a$, $x \mapsto \alpha a+\beta b$, we obtain
\begin{equation}
[\alpha f(b) + \beta f(a) ] g(\alpha a+\beta b) = [\alpha g(b) + \beta g(a)] f(\alpha a+\beta b)
\end{equation}
or, equivalently,
\begin{equation}
\beta [g(a)f(\alpha a+\beta b) - f(a)  g(\alpha a+\beta b) ] = \alpha [ f(b)g(\alpha a+\beta b)  -  g(b) f(\alpha a+\beta b) ]
\end{equation}
for all $a,b \in I$. By the definition of the function $v$ we obtain that
\begin{equation}
\label{vdefine}
\beta g(a)(v(\alpha a+\beta b)-v(a))=\alpha g(b)(v(b)-v(\alpha a+\beta b))
\end{equation}
for all $a,b \in I$. The proof of Lemma~\ref{lem:elem} is thus completed.
\end{proof}

\begin{lem}\label{lemma:set}
	If $v$ is not a constant function on $I$, then
$$
\{x \in I \colon \exists\, c \in \R \colon v |_{(x, \sup(I))} = c\} = \emptyset.
$$
\end{lem}
\begin{proof}
We prove Lemma~\ref{lemma:set} by contradiction. For this we assume that
$$
S :=\{x \in I \colon \exists\, c \in \R \colon v |_{(x, \sup(I))} = c\} \neq \emptyset,
$$
then, there exists $t_0 \in S$ and let's denote $k:=\inf(S)$.
Note that $k\in [\inf (I), \sup (I))$. First, assume that $k>\inf (I)$. Then, there exists $b_0\in I$ and sufficiently small $\varepsilon>0$ such that
\begin{equation}\label{5}
\sup (I)>b_0>\max \{ k+\tfrac{\alpha}{\beta}\cdot \varepsilon, k+\varepsilon \} \quad \text{and} \quad k-\varepsilon \in I.
\end{equation}
This and Lemma~\ref{lem:elem} imply
for all $x\in (k-\varepsilon,k+\varepsilon)$  that
\begin{equation}
\label{eq:lem:beta:g}
\beta g(x)(v (\alpha x+\beta b_0)-v(x))=\alpha g(b_0)(v(b_0)-v (\alpha x+\beta b_0)).
\end{equation}
From \eqref{5} we have for all $x\in (k-\varepsilon,k+\varepsilon)$ that
\begin{equation}
\sup (I)>b_0>\alpha x+\beta b_0>\alpha(k-\varepsilon)+\beta b_0>k
\end{equation}
and therefore, $v(b_0)=v (\alpha x+\beta b_0)=v |_{(k,\sup (I))}$. This together with \eqref{eq:lem:beta:g} and the assumption that $\forall \, x \in I \colon g(x) \neq 0$
ensures for all $x\in (k-\varepsilon,k+\varepsilon)$ that
 $v(x)=v (\alpha x+\beta b_0)=v |_{(k,\sup (I))}$. Hence, we obtain $v|_{(k-\varepsilon, \sup (I))}\equiv const$ which contradicts to $k = \inf(S)$. Therefore, we get that $k=\inf(I)$. This and the fact that  $S$ is connected prove
that $S=I$, which, in turn, implies that $v$ is constant on $I$. This contradicts to the assumption that $v$ is not a constant function on $I$. The proof of Lemma~\ref{lemma:set} is thus completed.
\end{proof}

\begin{lem}
	\label{lem:A:function}
 Let $A \colon I \times I \to \R$ be a function such that $A(a,b)=v (\alpha a+\beta b)-v(a)$ for all $(a,b) \in I \times I$ and  assume that $v$ is not a constant function on $I$. Then for every $t\in I$,
  \begin{enumerate}[{\upshape(i)}]
  \item\label{item:A:1} there exists  $b_0 \in (t, \sup (I))$  such that $A(t,b_0)\neq 0$;
  \item\label{item:A:2}  there exist $b_0  \in (t, \sup (I))$ and $\varepsilon >0$ such that $t-\varepsilon, t+\varepsilon \in I$, $b_0 \in (t+\varepsilon, \sup (I))$ and $\forall \, x\in (t-\varepsilon, t+\varepsilon) \colon  A(x,b_0)\neq 0$.
\end{enumerate}
\end{lem}
\begin{proof}
The claim in~\eqref{item:A:2} follows from the continuity of the function $A$ and the claim in~\eqref{item:A:1}. Therefore, it is enough to prove the claim in~\eqref{item:A:1}. We prove it by contradiction and for this we assume that  there exists $t_0\in I$ such that for all $b\in (t_0, \sup(I))$ it holds that
\begin{align}
\label{eq:A:0}
A(t_0,b)=v(\alpha t_0+\beta b)-v(t_0)=0.
\end{align}
Observe that  Lemma~\ref{lem:elem} implies for all $b\in (t_0, \sup(I))$ that
\begin{equation}
\label{eq:A:transform}
\beta g(t_0)(v(\alpha t_0+\beta b)-v(t_0))=\alpha g(b)(v(b)-v(\alpha t_0+\beta b)).
\end{equation}
Since $t_0, b\in I$ and $g|_I\neq 0$, using \eqref{eq:A:0} and \eqref{eq:A:transform} we obtain  for all $b\in (t_0, \sup(I))$ that
\begin{equation}
v(b)-v(\alpha t_0+\beta b)=0.
\end{equation}
This implies that $v(t_0)=v(b)$  for all $b \in (t_0,\sup(I))$  and thus $v|_{(t_0, \sup(I))} \equiv const$, 
contradicting to Lemma~\ref{lemma:set} since $v$ is not a constant function on $I$. This concludes the proof of Lemma~\ref{lem:A:function}.
\end{proof}

\begin{prop}\label{pro1}
On the interval $I$, either $\{ 1,F,G\}$ are linearly dependent or both $F$ and $G$ are infinitely differentiable.
\end{prop}

\begin{proof}
If $\{ 1,F,G\}$ are linearly dependent, then \eqref{new-eq} holds.  Assume that $\{ 1,F,G\}$ are linearly independent on $I$. This implies that $v$ is not a constant function on $I$. Next note that \eqref{new-eq} together with the transformation $\alpha a+\beta b \mapsto x$, $b-a \mapsto h$ implies that
\begin{equation}\label{new-eq2}
  v(x)=\frac{f(x)}{g(x)}=\frac{F(x+\alpha h)-F(x-\beta h)}{G(x+\alpha h)-G(x-\beta h)}
\end{equation}
for all $x\in I$ and $h\in \R $ such that $x+\alpha h, x-\beta h\in I$. The assumption that $g$ does not vanish on $I$ proves that $G$ is injective on $I$ and hence, it can be seen from \eqref{new-eq2} that the function $v$ is differentiable on $I$. Next, note that Lemma~\ref{lem:A:function} implies that for every $t\in I$ there exists $b_0 \in (t,\sup(I))$ and $\varepsilon >0$ such that $t-\varepsilon, t+\varepsilon \in I$, $b_0 \in (t+\varepsilon, \sup (I))$ and $v(\alpha x+\beta b_0)-v(x)\neq 0$ for all $x\in (t-\varepsilon, t+\varepsilon)$. This and \eqref{new-eq3} (with $a=x$ and $b=b_0$) show that for all $x\in (t-\varepsilon,t+\varepsilon)$ it holds that
  \begin{equation}\label{6}
  g(x)=\frac{\alpha}{\beta} \cdot \frac{g(b_0)(v(b_0)-v(\alpha x+\beta b_0))}{v(\alpha x+\beta b_0)-v(x)}.
\end{equation}
Since $v$ is differentiable,
 \eqref{6}  ensures that  $g$ is differentiable at $t\in I$. As $t$ was chosen arbitrarily in $I$ we obtain that $g|_I$ is differentiable. Hence, $f|_I=v \cdot g|_I$  is also differentiable. So we get both $F|_I$ and $G|_I$ are twice differentiable. This together with \eqref{new-eq2} show that $v$ is twice differentiable and,  from \eqref{6}, so is $g|_I$. Hence, $f|_I=v \cdot g|_I$ is also twice differentiable. This implies both $F|_I$ and $G|_I$ are three times differentiable. Repeating this bootstrapping argument, we obtain by induction argument that $F$ and $G$ are infinitely differentiable on $I$. This completes the proof of Proposition~\ref{pro1}.
\end{proof}

\section{Passing a local information to a global one}
\label{sec3}

In this section we present some properties of the solutions $(F, G)$ of the functional-differential equation \eqref{eq:prob2}. For convenience, we assume the following setting throughout this and the next sections.

\begin{set}
	\label{setting2}
	Let $J \subseteq \R$ be a nonempty open interval, let $\alpha, \beta\in (0,1)$ satisfy $\alpha+\beta=1$, let $F,G \colon J \to \R$ be differentiable functions with derivatives $F' = f$, $G'=g$ and which satisfy for all $a, b \in J$ that
	\begin{equation}\label{eq:setting2}
		[F(b)-F(a)]g(\alpha a+\beta b)=[G(b)-G(a)]f(\alpha a+\beta b).
	\end{equation}
\end{set}

We introduce the sets
\begin{align}\label{U_f,U_g}
U_f := \{ x \in J \colon f(x) \ne 0 \}, \quad U_g := \{ x \in J\colon g(x) \ne 0 \},
\end{align}
and also their complements $Z_f := J \setminus U_f$ and $Z_g := J \setminus U_g$. Observe that if $U_g$ is empty, i.e. $G$ is constant on $J$, then \eqref{eq:setting2} holds for trivial reasons (both sides are identically zero) for any differentiable function $F$. Similarly, if $F$ is constant then \eqref{eq:setting2} holds for any differentiable function $G$. Assume therefore that $U_g\neq\emptyset$. Then there is a sequence of mutually disjoint open intervals $\{I_\sigma\}_{\sigma\in\Sigma}$, $\Sigma\subset\mathbb{N}$, such that
\begin{align}
\label{rep:U_g=sum.of.itvs}
\displaystyle U_g=\bigcup_{\sigma\in\Sigma}I_{\sigma}.
\end{align}
%

\begin{prop}
\label{prop:cauchy.gen.U_g<>empty}
If $U_g\neq\emptyset$ but $U_f\cap U_g = \emptyset$ and $J$ is semi-infinite interval, then $U_f=\emptyset$, i.e. $f\equiv 0$ on $J$ and thus $F$ is constant.
\end{prop}
\begin{proof}
Assume that $\sup( J)=+\infty$. Since $U_g\neq\emptyset$, there is a non-empty interval $(p,q)\subset U_g$ such that $g(x)\neq 0$ on $(p,q)$ with $p>\inf (J)$ and $q<\sup (J)$ (otherwise we can choose ($p+\varepsilon, q-\varepsilon$) for some $0<\varepsilon<\frac{q-p}{2}$). Hence, $f(x)=0$ for all $x\in[p,q]$.
Then with the change of variables $(b-a) \mapsto h$, $(\alpha a+\beta b) \mapsto x$, \eqref{new-eq} yields
\begin{equation} \label{constancy}
F(x+\alpha h) - F(x-\beta h) = 0,
\end{equation}
for all $x\in [p,q]$, $h \in \R$ such that $x+\alpha h, x-\beta h \in J$.
%
%
%
We fix $x=q$ and choose~$h$ arbitrarily in $[0,\frac{q-p}{\beta}]$ so that $x-\beta h=q-\beta h \in [p,q]$ and $x+\alpha h = q+\alpha h  \in [q,q+\frac{\alpha}{\beta}(q-p)]$. Therefore, it follows from \eqref{constancy} that
\begin{equation*}
 F|_{[p,q+\frac{\alpha}{\beta}(q-p)]}=F|_{[q,q+\frac{\alpha}{\beta}(q-p)]}=F|_{[p,q]}\equiv const.
\end{equation*}
Similarly, if we fix $x=p$ and choose $h$ arbitrarily in $[0,(\frac{1}{\alpha}+\frac{1}{\beta})(q-p)]$, then $x+\alpha h=p+\alpha h \in [p,q+\frac{\alpha}{\beta}(q-p)]$ and 
$$x-\beta h=p-\beta h \in \Bigl(\max\Bigl\{\inf (J); p-(q-p)\Bigl(\frac{\beta}{\alpha}+1\Bigr)\Bigr\},p\Bigr].$$
This together with the last equality and \eqref{constancy} imply that
$$
F|_{\bigl(\max\bigl\{\inf (J); p-(q-p)\bigl(\frac{\beta}{\alpha}+1\bigr)\bigr\},\,q+\frac{\alpha}{\beta}(q-p)\bigr]}=F|_{[p,q]}\equiv const.
$$
It can be seen that, every time the constancy interval increases from below at least by the constant $(\frac{\beta}{\alpha}+1)(q-p)$, until we reach $\inf(J)$, and therefore repeating this technique, we eventually obtain
\begin{equation*}
  F|_{(\inf (J),k_1]}=F|_{[p,q]}\equiv const
\end{equation*}
for some $k_1>q$. Next let us consider the set
$$
S:=\{ x\in (q, +\infty) \colon F|_{(\inf (J),x]}=F|_{[p,q]}\equiv const\}.
$$
Since $k_1\in S$,  it is non-empty. Assume $k=\sup( S)<\sup (J)=+\infty$. Then $F|_{(\inf (J),k]}=F|_{[p,q]}$ and $f\neq 0$ in $(k,r)$ for some $r$.
By the assumption $U_f\cap U_g = \emptyset$, we get $g(x)=0$ for all $x\in [k,r]$. Using the above argument for $g$, we obtain that $G|_{(\inf (J),l]}\equiv const$ for some $l\ge r>k\ge q$. This, in turn, contradicts to the fact that $g\neq 0$ on $(p,q)\subset (\inf (J),l]$. Hence, it follows that $k=\sup (S)=+\infty$ and thus $f\equiv 0$ on $J$. The case $\inf (J)=-\infty$ can be analysed analogously.
\end{proof}

\begin{ex}\label{ex1}
  Let $J=(0,1)$ and consider the functions
\begin{equation*}
  F(x)=\left\{
         \begin{array}{ll}
           c_1, & \quad \hbox{$x \in (0,\frac{4}{5}]$}, \\[1ex]
           (x-\frac{4}{5})^2+c_1,  & \quad \hbox{$x \in (\frac{4}{5},1)$},
         \end{array}
         \right.
\end{equation*}
\begin{equation*}
  G(x)=\left\{
         \begin{array}{ll}
           (x-\frac{2}{5})^2+c_2, & \quad  \hbox{$x \in (0,\frac{2}{5})$}, \\[1ex]
           c_2, & \quad \hbox{$x \in [\frac{2}{5},1)$}.
         \end{array}
       \right.
\end{equation*}
\end{ex}

\noindent
It is easy to check that such $J$, $F$, and $G$ satisfy the equation \eqref{eq:setting2} for $\alpha=\beta=\frac{1}{2}$. Moreover, $U_g=(0,\frac{2}{5})\neq \emptyset$ and $U_g\cap U_f=\emptyset$ but $U_f=(\frac{4}{5},1)\neq \emptyset$.
Hence, the example shows that the statement of Proposition~\ref{prop:cauchy.gen.U_g<>empty} does not hold  if $J$ is a finite interval.

If $J$ is a semi-infinite interval, then Proposition~\ref{prop:cauchy.gen.U_g<>empty} implies that $U_f\cap U_g=\emptyset$ only when $U_f=\emptyset$ or $U_g=\emptyset$. From this fact we conclude the result below.

\begin{prop}
\label{prop:X}
Assume that
\begin{align}
\label{Uf_disj_Ug}
U_f\cap U_g \neq \emptyset
\end{align}
and consider the representation \eqref{rep:U_g=sum.of.itvs}. If $\inf (J)=-\infty$ or $\sup (J)=+\infty$ and $\{F, G, 1\}$ are linearly dependent as functions on $I_{\sigma}$ for every $\sigma\in\Sigma$, then $\{F, G, 1\}$ are linearly dependent on $J$.
\end{prop}
\begin{proof}
Assume that $\sup (J)=+\infty$. For $\sigma_1, \sigma_2 \in \Sigma$ with $\sigma_1 \neq \sigma_2$, consider the intervals $I_{\sigma_1} := (p_1,q_1)$, $I_{\sigma_2}:=(p_2,q_2)$ with
\begin{align*}
p_1<q_1\leq p_2<q_2,
\end{align*}
and assume that $\{F, G, 1\}$ are linearly dependent on $I_{\sigma_1}$ and $I_{\sigma_2}$.
%
%
%
Then it follows that there are constants $c_1, c_2\in\R$ such that
\begin{align}\label{FA_1B_2}
f(x) = \begin{cases}
c_1g(x) \colon & x\in I_{\sigma_1} \\
c_2g(x)\colon & x\in I_{\sigma_2}
\end{cases}.
\end{align}
With the change of the variables $(b-a) \mapsto h$, $(\alpha a+\beta b) \mapsto x$, \eqref{eq:setting2} yields that
\begin{equation*}
[F(x+\alpha h)-F(x-\beta h)]\, g(x) = [G(x+\alpha h)-G(x-\beta h)]\, f(x)
\end{equation*}
for all $x, h\in\R$ such that $x+\alpha h, x-\beta h \in J$. Using the fact that $g$ does not vanish on $I_{\sigma_2}$ and \eqref{FA_1B_2}  we obtain for all $x\in I_{\sigma_2}$, $h>0$ with $x+\alpha h$, $x-\beta h \in J$ that
\begin{equation}
\label{eq:dependent}
  F(x+\alpha h)-F(x-\beta h) = c_2[G(x+\alpha h)-G(x-\beta h)].
\end{equation}
Next, let us fix $x\in I_{\sigma_2}$. Then there exists $h>0$ such that $x-\beta h \in I_{\sigma_1}$ and that $x+\alpha h \in J$. Denoting $x+\alpha h$ by $y$ and  differentiating \eqref{eq:dependent}  with respect to $h$ we get
\begin{align}\label{FFA_2}
f(y-h) = c_2g(y-h).
\end{align}
Hence, from \eqref{FA_1B_2} and \eqref{FFA_2} we have
\begin{equation*}
  0=(c_1-c_2)g(y-h).
\end{equation*}
But $y-h\in I_{\sigma_1}$, so $g(y-h)\neq 0$ and thus
\begin{align}\label{A_1=A_2}
c_2-c_1=0.
\end{align}
Since $\sigma_1, \sigma_2\in\Sigma$ were arbitrary, \eqref{A_1=A_2} together with \eqref{FA_1B_2}  imply
\begin{align}\label{fAg}
f(x)=cg(x) \quad \text{for some constant} \quad c\in\R \quad \text{and all} \quad x\in U_g.
\end{align}
On the other hand, by changing the roles of $F$ and $G$ in the above analysis, we come to the conclusion that
\begin{align}\label{gKf}
g(x)=kf(x) \quad \text{for some constant} \quad k\in\R \quad \text{and all} \quad x\in U_f.
\end{align}
By \eqref{Uf_disj_Ug} there is an $x_0\in U_g\cap U_f$ and hence, $ck=1$. Therefore,  $c \neq 0$ and $k \neq 0$.  But then \eqref{fAg} implies  $U_g\subseteq U_f$ and \eqref{gKf} implies $U_f\subseteq U_g$. Therefore, $U_g=U_f$ and $Z_g=Z_f$. The latter means that for $x\in Z_g=Z_f$ we have that $f(x)=0= cg(x)$ and $g(x)=0=kf(x)$. Hence, with \eqref{fAg} and \eqref{gKf} these identities are valid on the entire $J=U_f\cup Z_f=U_g\cup Z_g$. In particular, it follows that $\{F, G, 1\}$ are linearly dependent on $J$. The case $\inf (J)=-\infty$ can be analysed analogously.
\end{proof}

\section{Main result for the asymmetric case}
\label{sec4}

In this section we consider the asymmetric case, i.e. we assume in \eqref{eq:setting2} that
\begin{align}\label{c}
\alpha, \, \beta \in (0,1) \quad \text{with } \quad \alpha\neq 1/2 \quad \text{and} \quad \beta=1-\alpha.
\end{align}
The following proposition describes all pairs $(F,G)$ of differentiable functions satisfying \eqref{eq:setting2} in the intervals where $g=G'$ does not vanish under the conditions~\eqref{c}.

\begin{prop}\label{cauchy1}
Let $(F,G)$ be a solution of  Problem~{\upshape\ref{3}} with $\alpha,\beta$ satisfying \eqref{c} and let $I=(p,q)\subseteq J $ be an interval where the derivative $G'$ does not vanish. Then $\{F, G, 1\}$ are linearly dependent on $I$.
\end{prop}

Proposition~\ref{pro1} implies that if $\{F,G,1\}$ are linearly independent on $I$, then $F$ and $G$ are infinitely differentiable on $I$. Combining this with \cite[Proposition~6]{BIM} we conclude the proof of Proposition~\ref{cauchy1}. 
We note that, although \cite[Proposition~6]{BIM} is formulated only to the case $J=\R$, the result can easily be generalized to an arbitrary open interval $J \subseteq \R$ just by substituting $\R$ to $J$ along the lines in the proof of \cite[Proposition~6]{BIM}.

The following theorem is the main result of this section.

\begin{thm}\label{thcauchy1}
Let $(F,G)$ be a solution of Problem~{\upshape\ref{3}} with $\alpha,\beta$ satisfying \eqref{c}. If $\inf (J)=-\infty$ or $\sup (J)=+\infty$, then $\{F, G, 1\}$ are linearly dependent on $J$, i.e.~there exist constants $c_1,c_2,c_3\in\R$ such that not all of them are zero and
\begin{equation}\label{lin.dep.1}
c_1F(x)+c_2G(x)+c_3=0 \quad \text{for all} \quad x\in J.
\end{equation}
\end{thm}
\begin{proof}
We consider three cases separately.

\medskip
\noindent
Case 1: $U_g = \emptyset$.
In this case $G$ is a constant on $J$ and \eqref{eq:setting2} holds for any differentiable function $F$. Hence \eqref{lin.dep.1} holds, for example, with $c_1=0$, $c_2=1$, $c_3=-G$ and thus $\{F, G, 1\}$ are linearly dependent on $J$.

\medskip
\noindent
Case 2: $U_g \neq \emptyset$ but $U_g \cap U_f=\emptyset$.
In this case Proposition~\ref{prop:cauchy.gen.U_g<>empty} yields that $F$ is a constant on $J$ and \eqref{eq:setting2} holds for any differentiable function $G$. Hence \eqref{lin.dep.1} holds, for example, with $c_1=1$, $c_2=0$, $c_3=-F$ and thus $\{F, G, 1\}$ are again linearly dependent on $J$.

\medskip
\noindent
Case 3: $U_g \cap U_f \neq \emptyset$.
In this case Propositions~\ref{cauchy1} and \ref{prop:X} directly imply that $\{F, G, 1\}$ are linearly dependent on ~$J$.
\end{proof}

\begin{rem}
   If $J$ is a bounded interval, then $\{1,F,G\}$ do not need to be linearly dependent on the whole of $J$. We refer the reader to the Example~\ref{ex1} above and also to \cite[Examples~13 and 14]{Lukasik}.
\end{rem}

\section{Main result for the symmetric case and the proof of Theorem~\ref{thm:sym.cauchy.main}}
\label{sec5}

In this section we consider the problem of describing all pairs $(F,G)$ of differentiable functions for which the mean value in \eqref{eq:setting2} is taken at the midpoint of the interval. Our first result gives a necessary (and also sufficient in the case when $\{F,G,1\}$ are linearly independent) condition on such pairs in the intervals where $g=G'$ does not vanish.

\begin{prop}
\label{cauchy2}
Assume that $F,G\colon J \to \R$ are differentiable functions with derivatives $F'=f$, $G'=g$. Let $I\subseteq J$ be  an interval such that $g\neq 0$ for all $x\in I$ and \eqref{eq:setting2} holds
for all $a,b \in I$. Then there exist constants $A, K\in\R$ and $x_{0}\in I$ such that
\begin{equation} \label{f,g}
f(x) = \bigg(A + K \int_{x_{0}}^{x}\frac{dt}{g^{2}(t)}\bigg) \, g(x), \quad \text{for all} \quad x \in I.
\end{equation}

\noindent
Moreover, if \eqref{f,g} holds with $K\neq0$, then on $I$, $g$ and thus $f$ have one of the forms of $Px+Q, Pe^{\mu x}+Qe^{-\mu x}$ and  $P\sin(\mu x)+Q\cos(\mu x)$, where $P,Q$ are real constants and $\mu >0$.
\end{prop}

If $\{1,F,G\}$ are linearly dependent on $I$, we choose $K=0$ in \eqref{f,g}. If $\{1,F,G\}$ are linearly independent on $I$, Proposition~\ref{pro1} implies that $F$ and $G$ are infinitely differentiable on $I$. Using this and \cite[Proposition~8]{BIM} we conclude the proof of Proposition~\ref{cauchy2}. Similarly to our last remark above, strictly speaking, \cite[Proposition~8]{BIM} is applicable only to the case $J = \R$. However, one can  see that this result can easily be generalized to an arbitrary $J \subseteq \R$ just by substituting $\R$ to $J$ along the lines in the proof of \cite[Proposition~8]{BIM}. Hence, on $I$ the function  $G$ has one of the following forms
\begin{align}\label{G_quad}
G(x)&=c_1x^2+c_2x+c_3,\\\label{G_exp}
G(x)&=c_1e^{\mu x}+c_2e^{-\mu x}+c_3, \quad \mu>0,\\\label{G_trig}
G(x)&=c_1\sin(\mu x)+c_2\cos(\mu x)+c_3, \quad \mu>0,
\end{align}
where $c_1,c_2,c_3$ are real constants. Altogether, we come to the following conclusion.

\begin{rem}
\label{rem:cauchy.sym}
On every interval $I\subseteq J$ on which $G'\neq 0$, either $\{F, G, 1\}$ are linearly dependent, or $G$ and thus also $F$ has one of the forms described in \eqref{G_quad}--\eqref{G_trig}.
\end{rem}
In the sequel, we call a function $G$ (respectively the pair $(F,G)$) to be of \textit{quadratic, exponential} or \textit{trigonometric type }on $I$ if $G$ has (respectively both of $F$ and $G$ have) the form \eqref{G_quad}, \eqref{G_exp} or \eqref{G_trig}, respectively.

\begin{lem}\label{qet.qet}
If $(F,G)$ are of quadratic type on some $I_{\sigma}\subset U_g$, then they are of quadratic type on the whole of $J$. This statement holds also for functions of exponential and trigonometric types.
\end{lem}
\begin{proof}
Note from Proposition~\ref{pro1} that linearly independent pairs of functions $(F,G)$ can be of \textit{quadratic, exponential} or \textit{trigonometric type} on some $I_{\sigma}\subset U_g$ only when $\alpha=\beta=1/2$. Let $I=(p,q)\subset U_g$ be an interval where $(F,G)$ are of \textit{quadratic type}. Substituting $\frac{a+b}{2} \mapsto x, \frac{b-a}{2} \mapsto h$, \eqref{eq:setting2} reads as
\begin{equation}\label{qet}
[F(x+h)-F(x-h)]g(x)=[G(x+h)-G(x-h)]f(x),
\end{equation}
for all $x\in (p,q)$, $h\in \R$ such that $x-h, x+h\in J$. Without loss of generality we may assume that $G(x)=c_1x^2+c_2x+c_3$ with $c_1\neq 0$ for $x\in I$. Then we set from \eqref{f,g} that $F(x)=c_4G(x)+c_5x+c_6$ for $x\in I$. Putting these to \eqref{qet}, we obtain
\begin{equation*}
F(x+h)-F(x-h)=[G(x+h)-G(x-h)]\bigl(c_4+\tfrac{c_5}{g(x)}\bigr)
\end{equation*}
for all $x\in (p,q)$ and $h\in \R$ such that $x-h, x+h\in J$. Differentiating both sides with respect to $h$ and $x$ simultaneously gives
\begin{equation*}
f(x+h)+f(x-h)=(g(x+h)+g(x-h))\big(c_4+\tfrac{c_5}{g(x)}\big)
\end{equation*}
and
\begin{center}
$f(x+h)-f(x-h)=(g(x+h)-g(x-h))\big(c_4+\frac{c_5}{g(x)}\big)-(G(x+h)-G(x-h))g'(x)\frac{c_5}{g(x)^2}$
\end{center}
for all $x\in (p,q)$ and $h\in \R$ such that $x+h,x-h \in J$. Using the last two equations we get
\begin{center}
$2f(x-h)=2g(x-h)\bigl( c_4+\frac{c_5}{g(x)}\bigr)+(G(x+h)-G(x-h))g'(x)\frac{c_5}{g(x)^2}$
\end{center}
for all $x\in (p,q)$ and $h\in \R$ such that $x-h, x+h\in J$. We fix $0<\varepsilon <\frac{q-p}{2}$ and choose $h=\varepsilon$. Then $x-h \in (p,q)$ and $x+h\in (q,q+\varepsilon)$ for all $x\in (q-\varepsilon,q) \subset (p,q)$. Hence, from the last equation we obtain $G(x+h)=c_1(x+h)^2+c_2(x+h)+c_3$ for all $x\in (q-\varepsilon,q) \subset (p,q)$. This together with \eqref{f,g} imply that $(F,G)$ are of \textit{quadratic type} on $(p,q+\varepsilon)$. The critical points of $F$ and $G$ cannot disturb us because every function of \textit{quadratic type} is continuous and has at most one critical point. Hence, using the above argument, we can extend the interval $(p,q)$ to the whole of $J$.

All functions of \textit{exponential type} has also at most one critical point, thus the above argument apply for them as well. Every function of \textit{trigonometric type} is continuous and has a finite number of critical points in every finite subinterval $(p,q)$ of $J$. Hence, using the above argument, we again complete the proof.
\end{proof}


Now we are ready to  give the proof of Theorem~\ref{thm:sym.cauchy.main}.

\paragraph{Proof of Theorem~\ref{thm:sym.cauchy.main}.} First, we describe the solutions of Problem~\ref{3}. Consider the set $U_g$ defined in \eqref{U_f,U_g}. If $U_g=\emptyset$, then $g\equiv 0$ on $J$, and thus $G$ is constant on $J$. In this case $F$ can be an arbitrary differentiable function on $J$ and thus $\{F,G,1\}$ are linearly dependent on $J$. Next, let us assume that $U_g\neq \emptyset$ and consider the representation \eqref{rep:U_g=sum.of.itvs}. If $\alpha \neq 1/2$, then Proposition~\ref{cauchy1} and Theorem~\ref{thcauchy1} imply that $\{F,G,1\}$ are linearly dependent on every interval $I\subset U_g$ or if $J$ is semi-infinite interval, they are linearly dependent on the whole of $J$. So, we can focus on the symmetric case of $\alpha=\beta=1/2$. It is clear that if $\{F,G,1\}$ are linearly independent on $I_{\sigma}$ for some $\sigma \in \Sigma$, then Remark~\ref{rem:cauchy.sym} and Lemma~\ref{qet.qet} imply that $F,G$ have only one of the \textit{quadratic, exponential} or \textit{trigonometric types} on whole of $J$. Otherwise, we are left only with the case $\{F,G,1\}$ being linearly dependent on every interval $I\subset U_g$ which was already considered above.
The solutions of the Problem~\ref{3} directly yield the proof of the Theorem~\ref{thm:sym.cauchy.main}. \qed

\smallskip



\bibliographystyle{acm}
\bibliography{bibfile}

\end{document}